\documentclass[12pt,twoside]{article}
\usepackage{a4}
\usepackage{amssymb,amsmath,amsthm,latexsym}
\usepackage{amsfonts}
\newtheorem{thm}{Theorem}[section]
\newtheorem{cor}[thm] {Corollary}
\newtheorem{lem} [thm]{Lemma}
\newcommand\dsum{\displaystyle\sum}
\newcommand\Eta{\mathrm{H}}
\newcommand\transpose{^{\mathrm T}}
\newcommand\NEPS{\operatorname{NEPS}}

\newcommand\cB{{\mathcal B}}
\newcommand\bj{{\mathbf j}}
\newcommand\bk{{\mathbf k}}
\newcommand\bu{{\mathbf u}}
\newcommand\bv{{\mathbf v}}
\newcommand\bbZ{{\mathbb Z}}
\newcommand\Cvet{\'Cvetkovi\'c }
\allowdisplaybreaks
\newcommand\mylabel[1]{\label{#1}}
\begin{document}

\markboth {\hspace*{-9mm} \footnotesize \sc
  Products, Line Graphs, Signed Graphs: Eigenvalues and Energy
                 }
                { \footnotesize \sc
         K.A.\ Germina,  Shahul Hameed K  and  Thomas Zaslavsky  \hspace*{-9mm}
               }

\title{\textbf{On Products and Line Graphs of Signed Graphs, their Eigenvalues and Energy}}
\thispagestyle{empty}
\date{\today}

\author{
K.A.\ Germina,\footnote{Kannur University Mathematics Research Center \& PG Department of Mathematics, Mary Matha Arts \& Science College, Vemom P.O., Mananthavady - 670645, INDIA.  \textbf{e-mail:} {\tt srgerminaka@gmail.com}} \
Shahul Hameed K\footnote{Kannur University Mathematics Research Center \& PG Department of Mathematics, Mary Matha Arts \& Science College, Vemom P.O., Mananthavady - 670645, INDIA.  \textbf{e-mail:} {\tt shabrennen@gmail.com}} \
and 
Thomas Zaslavsky\footnote{Department of Mathematical Sciences, Binghamton University (SUNY), Binghamton, NY 13902-6000, U.S.A.  \textbf{e-mail:} {\tt zaslav@math.binghamton.edu}.  Zaslavsky expresses his gratitude to the other authors for inviting him to assist with this paper.}
}
\maketitle
\vspace{0.50cm}

\begin{abstract}
In this article we examine the adjacency and Laplacian matrices and their eigenvalues and energies of the general product (non-complete extended $p$-sum, or NEPS) of signed graphs.  We express the adjacency matrix of the product in terms of the Kronecker matrix product and the eigenvalues and energy of the product in terms of those of the factor signed graphs.  For the Cartesian product we characterize balance and compute expressions for the Laplacian eigenvalues and Laplacian energy.  We give exact results for those signed planar, cylindrical and toroidal grids which are Cartesian products of signed paths and cycles.

We also treat the eigenvalues and energy of the line graphs of signed graphs, and the Laplacian eigenvalues and Laplacian energy in the regular case, with application to the line graphs of signed grids that are Cartesian products and to the line graphs of all-positive and all-negative complete graphs. 
\\
\end{abstract}

\textbf{Key Words:} Signed graph, Cartesian product graph, Line graph, Graph Laplacian, Kirchhoff matrix, Eigenvalues of graphs, Energy of graphs.

\textbf{Mathematics Subject Classification (2010):} \ Primary 05C50; \ Secondary 05C22, 05C76.

\tableofcontents
\setcounter{tocdepth}{3}

\newpage

\section{Introduction}\mylabel{introduction}
We study the adjacency and Laplacian (or Kirchhoff) matrices and their eigenvalues and energies of signed graphs that are Cartesian products, \Cvet products (generally called NEPS) or line graphs.  (All graphs in this article are simple and loop-free.)  

{\it Signed graphs} (also called sigraphs), with positive or negative labels on the edges, are much studied in the literature because of their use in modeling a variety of physical and socio-psychological processes (see \cite{bar} and \cite{bda, gbda}) and also because of their interesting connections with many classical mathematical systems (see \cite{tz3}).
Formally,  a signed graph is an ordered pair \ $\Sigma =(G, \sigma)$ \ where \ $G = (V,E)$ \ is a graph called the \textit{underlying graph} of \ $\Sigma$ \ and \ $\sigma: E \rightarrow \{+1,-1\}$, \ called a {\it signing} (also called a {\it signature}), is a function from the edge set \ $E$ \ of \ $G$ \ into the set \ $\{+1,-1\}$ \ of {\it signs}.  The sign of a cycle in a signed graph is the product of the signs of its edges. Thus a cycle is positive if and only if it contains even number of negative edges.  A signed graph \ $\Sigma$ \ is said to be {\it balanced} (or {\it cycle balanced}) if all of its cycles are positive.

A signed graph is {\it all-positive} (respectively, {\it all-negative}) if all of its edges are positive (negative); further, it is said to be {\it homogeneous} if it is either all-positive or all-negative. 
A graph can be considered to be a homogeneous signed graph; thus signed graphs become a generalization of graphs.  

The Cartesian product \ $\Sigma_{1}\times \Sigma_{2}$ \ of two signed graphs \ $\Sigma_{1} =(V_{1},E_{1}, \sigma_{1})$ \ and \ $\Sigma_{2} =(V_{2},E_{2}, \sigma_{2})$ \ is a generalization of the Cartesian product of ordinary graphs (see \cite[Section 2.5]{spec1}).  It is defined as the signed graph \ $(V_{1}\times V_{2},E, \sigma)$ \ where the edge set \ $E$ \ is that of the Cartesian product of underlying unsigned graphs and the signature function \ $\sigma$ \ for the labeling of the edges is defined by
\begin{equation*}
\sigma\big((u_{i},v_{j})(u_{k},v_{l})\big)= 
\begin{cases}
\sigma_{1}\big(u_{i}u_{k}) & \text{if } j=l,\\
\sigma_{2}\big(v_{j}v_{l})& \text{if } i = k.\\
\end{cases} 
\end{equation*}

In this paper, we treat the adjacency matrix and the Laplacian matrix of a signed graph.  These matrices are immediate generalizations of familiar matrices from ordinary, unsigned graph theory (\cite{spec1}). Thus, if \ $\Sigma =(G, \sigma)$ \ is a signed graph where \ $G=(V,E)$ \ with \ $V=\{v_1,v_2,\dots,v_n\}$, \ its adjacency matrix \ $A(\Sigma) = (a_{ij})_{n \times n}$ \ is defined as 
\begin{equation*}
a_{ij}= \begin{cases}  
    \sigma(v_{i}v_{j}), & \text{if \ $v_i$ \ and \ $v_j$ \ are adjacent},\\
    0, & \text{otherwise.}
\end{cases} 
\end{equation*}
The {\it Laplacian matrix} (or {\it Kirchhoff matrix} or {\it admittance matrix}) of a signed graph \ $\Sigma$, \ denoted by \ $L(\Sigma)$ \ (or \ $K(\Sigma)$), \ is \ $D(\Sigma)-A(\Sigma)$ \ where \ $D(\Sigma)$ \ is the diagonal matrix of the degrees of vertices of \ $\Sigma$. 

We treat two kinds of operation on signed graphs: the Cartesian product (for which we get the strongest results) and the class of generalizations called ``NEPS'' (or as we prefer ``\Cvet products'') introduced by \Cvet (\cite{cvet1}), and also the line graphs of signed graphs.

The ordinary adjacency and Laplacian matrices of a graph \ $G$ \ are identical with those of the all-positive signed graph \ $+G$.  \ The so-called signless Laplacian of \ $G$ \ (\cite{cvet}) is the Laplacian matrix of the all-negative graph \ $-G$.  \ Eigenvalues of the adjacency matrix, the Laplacian matrix and the signless Laplacian matrix of a graph have been widely used to characterize properties of a graph and extract some useful information from its structure. The eigenvalues of the adjacency matrix of a graph are often referred to as the {\it eigenvalues of the graph} and those of the Laplacian matrix as the {\it Laplacian eigenvalues}.  

Denote the eigenvalues of a matrix \ $M$ \ of order \ $n$ \ by \ $\lambda_j(M)$ \ for \ $j=1,2,\dots, n$.  \ The {\it energy} \ $E(\Sigma)$ \ of a signed graph \ $\Sigma$ \ is the sum of the absolute values of the eigenvalues of its adjacency matrix. The {\it Laplacian energy}\/ of \ $\Sigma$, \ denoted by \ $E_L(\Sigma),$ \ is defined as 
$$
E_L(\Sigma) = \sum_{j=1}^{|V|} |\lambda_j(L(\Sigma))-\bar d(\Sigma)| ,
$$ 
where \ $L(\Sigma)$ \ is the Laplacian matrix of \ $\Sigma$ \ and \ $\bar d(\Sigma) = 2|E|/|V|$ \ is the average degree of the vertices in \ $\Sigma$.  \ These definitions are direct generalizations of those used for unsigned graphs (\cite{rbk, igut1} for energy and \cite{igut3} for Laplacian energy).

\section{Preliminaries}\mylabel{preliminaries}

Many formulas in the examples have cases depending on the parity of a parameter.  Therefore, for an integer \ $r$, \ we define \ $[r] = 0$ \ if \ $r$ \ is even, \ $[r] = 1$ \ if \ $r$ \ is odd.

For a signed graph \ $\Sigma$, \ the quantity \ $c(\Sigma) = c(G)$ \ is the number of connected components of the underlying graph and \ $c_{b}(G)$ \ is the number of its components that are bipartite.  The quantity \ $b(\Sigma)$ \ is the number of connected components of \ $\Sigma$ \ that are balanced.  An essential lemma in signed graph theory is a characterization of balance by switching, which when expressed in terms of the adjacency matrix takes the following form:

\begin{lem}[\cite{tz1}]\label{tsw}
$\Sigma$ \ is balanced if and only if there is a diagonal matrix \ $S$ \ with diagonal elements \ $\pm1$ \ such that \ $SA(\Sigma)S$ \ is non-negative.  Then \ $SA(\Sigma)S = A(G)$ \ where \ $G$ \ is the underlying graph of \ $\Sigma$.
\end{lem} 

The {\it negation} of a signed graph \ $\Sigma=(G,\sigma)$, \ denoted by \ $-\Sigma = (G,-\sigma)$, \ is the same graph with all signs reversed.  The adjacency matrices are related by \ $A(-\Sigma) = -A(\Sigma)$.

\subsection{Rank and eigenvalues}

The {\it incidence matrix} of a signed graph \ $\Sigma$ \ with \ $n$ \ vertices and \ $m$ \ edges (\cite{tz1}) is the \ $n \times m$ \ matrix \ $\Eta(\Sigma) = \begin{bmatrix} \eta_{ij} \end{bmatrix}$ \ in which \ $\eta_{ik} = 0$ \ if \ $v_i$ \ is not incident with \ $e_k$, \ and \ $\eta_{ik} = \pm1$ \ if \ $v_i$ \ is incident with \ $e_k$, \ and such that for an edge \ $v_iv_j$, \ the product \ $\eta_{ik}\eta_{jk} = -\sigma(v_iv_j)$.  \ The incidence matrix is uniquely determined only up to multiplication of columns by \ $-1$, \ but that ambiguity does not affect any of the properties of interest to us.  In particular, the incidence matrix always satisfies the Kirchhoff equation \ $\Eta(\Sigma) \Eta(\Sigma)\transpose = L(\Sigma).$  \ For that reason the Laplacian matrix is positive semi-definite.

\begin{lem}[\cite{tz1}]\mylabel{llaprank} \quad 
The incidence matrix and the Laplacian matrix of a signed graph \ $\Sigma$ \ both have rank \ $n-b(\Sigma)$.

For a graph \ $G$, \ the Laplacian \ $L(G)=L(+G)$ \ has rank \ $n-c(G)$ \ and the signless Laplacian \ $Q(G)=L(-G)$ \ has rank \ $n-c_{b}(G)$.
\end{lem}

\begin{proof}
The rank of the incidence matrix is found in \cite{tz1}.  
The Laplacian matrix, being the product \ $\Eta(\Sigma) \Eta(\Sigma)\transpose$, \ has the same rank as \ $\Eta(\Sigma)$.  \ Because \ $+G$ \ is balanced, \ $b(+G)=c(G)$.  \ Because a component of \ $-G$ is balanced if and only if it is bipartite, \ $b(-G)$ \ is the number of bipartite components of \ $G$.
\end{proof}

Recall that the {\it spectrum} of a graph or signed graph is the spectrum of its adjacency matrix and that the spectrum is the list of eigenvalues with their multiplicities.  The {\it Laplacian spectrum} is the spectrum of the Laplacian matrix.  
Acharya's theorem, following, gives a spectral criterion for balance in signed graphs.  

\begin{thm}[\cite{gbda}]\mylabel{t1} \quad  
If  \ $\Sigma =(G, \sigma)$ \ is a signed graph, then \ $\Sigma$ \ is balanced if and only if \ $G$ \ and \ $\Sigma$ \ have the same spectrum.
\end{thm}

We take note of the special case in which the underlying graph \ $G$ \ is regular.  The following lemma generalizes the well known fact that, for a \ $k$-regular graph, the smallest eigenvalue is \ $-k$ \ occurring with multiplicity \ $c_b(G)$, \ the largest eigenvalue is \ $k$ \ with multiplicity \ $c(G)$ \ and the other eigenvalues fall into the open interval \ $(-k,k)$.

\begin{lem}\mylabel{lregular} \quad 
Assume \ $\Sigma$ \ has underlying graph \ $G$ \ which is regular of degree \ $k$.  \ Let the eigenvalues of \ $\Sigma$ \ be \ $\lambda_1,\lambda_2,\ldots,\lambda_n$ in weakly increasing order.  Then \ $\lambda_1,\ldots,\lambda_{b(-\Sigma)} = -k$, \ $-k < \lambda_{b(-\Sigma)+1},\ldots,\lambda_{n-b(\Sigma)} < k$ \ and \ $\lambda_{n-b(\Sigma)+1},\ldots,\lambda_n = k$.  \ The Laplacian eigenvalues of \ $\Sigma$ \ are \ $\lambda_i^L = k - \lambda_i$, \ including \ $2k$ \ with multiplicity \ $b(-\Sigma)$, \ $0 < \lambda_{b(-\Sigma)+1}^L,\ldots,\lambda_{n-b(\Sigma)}^L < 2k$, \ and \ $0$ \ with multiplicity \ $b(\Sigma)$.  \ The Laplacian energy equals the energy.
\end{lem}

\begin{proof}
The proof is by substituting in the definitions, using the facts that the degree matrix is \ $D = kI_n$, \ $A(-\Sigma) = -A(\Sigma)$ \ and \ $k$ \ is the average degree.
\end{proof}

Lemma~\ref{lregular} raises the question of whether it is possible to have \ $b(-\Sigma) > n-b(\Sigma)$, \ since if that is the case and \ $k\neq0$, \ then there is a contradiction in the notation of the lemma.  By Lemma~\ref{lexcessb}, a contradiction of that kind in Lemma~\ref{lregular} is not a problem because \ $b(-\Sigma) > n-b(\Sigma)$ \ implies \ $k=0$ \ and then all eigenvalues are \ $0=k=-k$.

\begin{lem}\mylabel{lbb} \quad 
Both \ $\Sigma$ \ and \ $-\Sigma$ \ are balanced if and only if the underlying graph \ $G$ \ is bipartite and \ $\Sigma$ \ or \ $-\Sigma$ \ is balanced.
\end{lem}

\begin{proof}
Let \ $C$ \ be a cycle in \ $G$.  \ The sign of \ $C$ \ in \ $-\Sigma$ \ equals \ $(-1)^{|C|}$ \ times the sign of \ $C$ \ in $\Sigma$.  \ Thus \ $C$ \ has the same sign in both \ $\Sigma$ \ and \ $-\Sigma$ \ if and only if it has even length.  Therefore, if \ $G$ \ contains an odd cycle, it is impossible for \ $\Sigma$ and \ $-\Sigma$ \ to both be balanced.  If \ $G$ \ is bipartite, then every cycle has the same sign in \ $\Sigma$~; therefore \ $\Sigma$ \ is balanced if and only if \ $-\Sigma$ \ is balanced.
\end{proof}

\begin{lem}\mylabel{lexcessb} \quad 
$b(-\Sigma) + b(\Sigma) \leq n$ \ except possibly when the number of isolated vertices is greater than the number of components with order at least \ $3$. 

In particular, if the underlying graph \ $G$ \ is \ $k$-regular and \ $b(-\Sigma) + b(\Sigma) > n$ \ then \ $k=0$.
\end{lem}

\begin{proof}
First let us consider a single connected component \ $\Sigma_i$ \ of \ $\Sigma$ \ whose order is \ $n_i$~; let \ $G_i$ \ be the corresponding component of \ $G$.  \ If \ $n_i\leq2$ \ then \ $b(\Sigma_i) = b(-\Sigma_i) = 1$, \ so \ $b(\Sigma_i) + b(\Sigma_i) \leq 2 = n_i+1$ \ if \ $n_i=1$ \ and \ $b(\Sigma_i) + b(\Sigma_i) \leq 2 = n_i$ \ if \ $n_i=2$.  \ If \ $n_i>2$ \ then \ $b(\Sigma_i) + b(\Sigma_i) \leq 2 < n_i$.  

Therefore if \ $b(-\Sigma) + b(\Sigma) > n$ \ then there are at least as many isolated vertices in \ $G$ \ as the number of components with order \ $3$ \ or greater.

When all vertices have the same degree \ $k$ \ there can be no isolated vertices unless \ $k=0$, \ which means there are no edges.
\end{proof}

\subsection{Kronecker product of matrices}

The identity matrix of order \ $n$ \ is denoted by \ $I_n$.  \ 
The Kronecker product of matrices \ $A=\begin{bmatrix} a_{ij} \end{bmatrix}_{m \times p}$ \ and \ $B$ \ of orders \ $m\times p$ \ and \ $n\times q$, \ respectively, is the matrix \ $A\otimes B$ \ of order \ $mn\times pq$ \ defined by 
$$A\otimes B = \begin{bmatrix}
   a_{11}B & a_{12}B & \dots & \dots & a_{1p}B \\
   a_{21}B & a_{22}B & \dots & \dots & a_{2p}B \\
   \dots & \dots& \dots & \dots & \dots  \\
   \dots & \dots& \dots & \dots & \dots  \\
   a_{m1}B & a_{m2}B & \dots & \dots & a_{mp}B
\end{bmatrix}.$$
The Kronecker product is a componentwise operation, that is, \ $(A\otimes B) (A'\otimes B') = (AA') \otimes (BB')$.  
It is also an associative operation; therefore a multiple product \ $A_1 \otimes A_2 \otimes \cdots \otimes A_\nu$ \ is well defined.  Let \ $A_i$ \ have order \ $m_i \times n_i$ \ and elements \ $a_{i;jk}$.  \ An element of such a product is indexed by a pair of \ $\nu$-tuples, a row index \ $\bj=(j_1,j_2,\ldots,j_\nu)$ \ and a column index \ $\bk=(k_1,k_2,\ldots,k_\nu)$, \ where \ $1 \leq j_i \leq m_i$ \ and \ $1 \leq k_i \leq n_i$.  \ The element \ $a_{\bj\bk}$ \ of the product matrix is 
\begin{equation}\label{ekronecker}
a_{\bj\bk} = a_{1;j_1k_1}a_{2;j_2k_2} \cdots a_{\nu;j_\nu k_\nu}.
\end{equation}

\begin{lem}[\cite{fz}] \mylabel{l1} \quad 
Let \ $A$ \ and \ $B$ \ be square matrices of orders \ $m$ \ and \ $n$, \ respectively, with eigenvalues \ $\lambda_{i}$ \ $(1\leq i\leq m )$ \ and \ $\mu_{j}$ \ $( 1\leq i\leq n )$.  \ Then the \ $mn$ \ eigenvalues of \ $A\otimes B$ \ are \ $\lambda_{i}\mu_{j}$, \ and those of \ $A\otimes I_{n}+ I_{m}\otimes B$ \ are \ $\lambda_{i}+ \mu_{j}$.
\end{lem}

The second part of the lemma is due to the fact that \ $A\otimes I_n$ \ and \ $I_{m}\otimes B$ \ are simultaneously diagonalizable.  The first part has an obvious extension to multiple products.  That is the first part of the next lemma.  The second part is the extension to multiple sums and products.

\begin{lem}\mylabel{leigen} \quad 
Let \ $A_i$, \ for each \ $i=1,\ldots,\nu$, \ be a square matrix of order \ $n_i$ \ and let \ $\lambda_{ij}$ \ for \ $1 \leq j \leq n_i$ \ be its eigenvalues.  Let \ $k_1,\ldots,k_\nu$ \ be non-negative integers.  Then the \ $n_1\cdots n_\nu$ \ eigenvalues of the Kronecker product \ $A_1^{k_1} \otimes \cdots A_\nu^{k_\nu}$ \ are \ $\lambda_{j_1\ldots j_\nu} = \lambda_{1j_1}^{k_1} \cdots \lambda_{\nu j_\nu}^{k_\nu}$ \ for \ $1 \leq j_i \leq n_i$.

Let \ $\bk_p = (k_{p1},\ldots,k_{p\nu})$ \ for \ $1 \leq p \leq q$ \ be vectors of non-negative integers.  \ Then the \ $n_1\cdots n_\nu$ \ eigenvalues of \ $\sum_{p=1}^q A_1^{k_{p1}} \otimes \cdots A_\nu^{k_{p\nu}}$ \ are \ $\lambda_{j_1\ldots j_\nu} = \sum_{p=1}^q \lambda_{1j_1}^{k_{p1}} \cdots \lambda_{\nu j_\nu}^{k_{p\nu}}$ \ for \ $1 \leq j_i \leq n_i$.
\end{lem}

\begin{proof}
The first part is obvious from Lemma~\ref{l1}.  The second part is true because the summed matrices commute, so they are simultaneously diagonalizable, they have the same eigenvectors and therefore their eigenvalues can be summed.
\end{proof}

\subsection{Products of signed graphs}

Now we define a general product of signed graphs following the idea of \Cvet for unsigned graphs (\cite{cvet1}) as described in \cite[Section 2.5]{spec1}.  We work with signed graphs \ $\Sigma_i = (V_i,E_i,\sigma_i)$, \ for \ $i=1,\ldots,\nu$, \ of order \ $n_i$, \ with underlying graph \ $G_i = (V_i,E_i)$, \ vertex set \ $V_i = \{v_{i1},v_{i2},\dots,v_{in_i}\}$ \ and adjacency matrix \ $A_i$.  \ We denote the eigenvalues of \ $\Sigma_i$ \ by \ $\lambda_{i1}, \lambda_{i2}, \ldots, \lambda_{in_i}$.  \ The Laplacian eigenvalues are denoted by a superscript \ $L$, \ as \ $\lambda_{ij}^L$.

The general product is known as the \emph{non-complete extended \ $p$-sum} or \emph{NEPS}, but we shall call it simply the \emph{\Cvet product}.  This product is defined in terms of a set \ $\cB$ \ of \ $0/1$ \ vectors, called the \emph{basis} for the product, such that for every \ $i \in \{1,2,\ldots,\nu\}$ \ there exists at least one \ $\beta \in \cB$ \ for which \ $\beta_i = 1$~; we say \ $\cB$ \ has \emph{support \ $\{1,2,\ldots,\nu\}$}.  \ First we define a product for one arbitrary vector \ $\beta = (\beta_1,\beta_2,\ldots,\beta_\nu) \in \{0,1\}^\nu \subset \bbZ^\nu$.  \ This product, written \ $\NEPS(\Sigma_1,\ldots,\Sigma_\nu;\beta)$, \ is the signed graph \ $(V,E,\sigma)$ \ with vertex set
$$
V = V_1 \times V_2 \times\cdots\times V_\nu,
$$
edge set
$$
E = \{ (u_{1},\dots,u_\nu)(v_{1},\dots,v_\nu) : u_i=v_i \text{ if } \beta_i = 0 \text{ and } u_iv_i \in E_i \text{ if } \beta_i = 1 \},
$$
and signature
$$
\sigma\big((u_{1},\dots,u_\nu)(v_{1},\dots,v_\nu)\big) = \prod_{i=1}^\nu \sigma_i(u_iv_i)^{\beta_i}
= \prod_{i: \beta_i=1} \sigma_i(u_iv_i).
$$
In the general definition we have a set \ $\cB = \{\beta_1,\ldots,\beta_q\} \subseteq \{0,1\}^\nu \setminus \{(0,0,\ldots,0)\}$ \ and we define 
$$
\NEPS(\Sigma_1,\ldots,\Sigma_\nu;\cB) 
= \bigcup_{\beta\in \cB} \NEPS(\Sigma_1,\ldots,\Sigma_\nu;\beta).
$$
The underlying graph of \ $\NEPS(\Sigma_1,\ldots,\Sigma_\nu;\cB)$ \ is the \Cvet product \ $\NEPS(G_1,G_2,\ldots,G_\nu; \cB)$ \ of the underlying graphs as defined in \cite[Section 2.5]{spec1}.

\begin{lem}\mylabel{ldisjoint}
If \ $\beta \neq \beta'$, \ then \ $\NEPS(\Sigma_1,\ldots,\Sigma_\nu;\beta)$ \ and \ $\NEPS(\Sigma_1,\ldots,\Sigma_\nu;\beta')$ \ have disjoint edge sets.
\end{lem}

\begin{proof}
For a vertex pair \ $\bu=(u_{1},\dots,u_{n}),\ \bv=(v_{1},\dots,v_{\nu})$, \ define \ $\beta(\bu,\bv)\in\{0,1\}^\nu$ \ by \ $\beta(\bu,\bv)_j = 0$ \ if \ $u_j=v_j$ \ and \ $1$ \ if \ $u_j\neq v_j$.  \ Then $\bu, \bv$ \ are adjacent in \ $\NEPS(\Sigma_1,\ldots,\Sigma_\nu;\beta)$ \ if and only if \ $\beta(\bu,\bv)_j = \beta_j$ \ for every \ $j$, \ i.e., \ $\beta(\bu,\bv) = \beta$.  \ This proves that \ $\bu\bv$ \ is an edge in \ $\NEPS(\Sigma_1,\ldots,\Sigma_\nu;\beta)$ \ for exactly one \ $\beta$.
\end{proof}

In particular, the Cartesian product \ $\Sigma_1 \times \Sigma_2 \times\cdots\times \Sigma_\nu$ \ arises by taking \ $\cB$ \ to be the set of all vectors with exactly one coordinate equal to \ $1$.  \ 
Another important product, called the \emph{strong product} or \emph{categorical product}, is obtained by taking \ $\cB = \{(1,1,\ldots,1)\}$.  \ A generalization of both, which could be called the \emph{symmetric \ $p$-sum} (but we think of it as a product), is obtained by taking the set \ $\cB_p$, \ for \ $1\leq p \leq q$, \ which consists of all vectors \ $\beta$ \ with exactly \ $p$ \ coordinates equal to \ $1$.  \ An incomplete \ $p$-product, where \ $\cB \subseteq \cB_p$, \ has the nice property that 
$$
\NEPS({-\Sigma_1},\ldots,-\Sigma_\nu;\cB) = (-)^p \NEPS(\Sigma_1,\ldots,\Sigma_\nu;\cB) .
$$
For instance in the Cartesian product \ $(-\Sigma_1) \times \cdots \times (-\Sigma_\nu) = -(\Sigma_1 \times \cdots \times \Sigma_\nu).$

A final property shows that a \Cvet product of all-positive signed graphs is essentially equivalent to the same product of the underlying graphs.  Clearly,
\begin{equation}\mylabel{eposneps}
\NEPS(+G_1,\ldots,+G_\nu;\cB) = +\NEPS(G_1,\ldots,G_\nu;\cB).
\end{equation}

\section{Main Results}\mylabel{main}

In this section we establish expressions for the adjacency, degree and Laplacian matrices of the \Cvet product  of signed graphs in terms of the Kronecker products of the corresponding matrices of the constituent graphs.  We also find similar expressions for the line graph of a signed graph.  We apply these results to calculate eigenvalues and energies in general and, in Section~\ref{examples}, for certain signed product graphs: planar, cylindrical and toroidal grids, and their line graphs.  An important application is the characterization of balance of the product graph.

\subsection{Adjacency matrix, eigenvalues and energy of products}
First we treat the adjacency matrix of the \Cvet product, which implies expressions for the eigenvalues.  For the eigenvalues of \ $\Sigma_i$ \ we write \ $\lambda_{ij}$, \ $1 \leq j\leq n_i$.  
This theorem generalizes \cite[Theorems 2.21 and 2.23]{spec1} to signed graphs.

\begin{thm}\mylabel{tadj} \quad 
Let \ $\Sigma = \NEPS(\Sigma_1,\ldots,\Sigma_\nu; \cB)$.  \ The adjacency matrix is given by 
$$
A(\Sigma) = \sum_{\beta\in \cB}  A_1^{\beta_1} \otimes \cdots \otimes A_\nu^{\beta_\nu} .
$$

The eigenvalues are 
$$
\lambda_{j_1\cdots j_\nu} = \sum_{\beta\in \cB}  \lambda_{1j_1}^{\beta_1} \cdots \lambda_{\nu j_\nu}^{\beta_\nu} 
$$
for \ $1 \leq j_1 \leq n_1, \ \ldots, \ 1 \leq j_\nu \leq n_\nu$.

The energy 
$$
E(\Sigma) = \sum_{j_1=1}^{n_1} \cdots \sum_{j_\nu=1}^{n_\nu}  \big| \sum_{\beta\in \cB}  \lambda_{1j_1}^{\beta_1} \cdots \lambda_{\nu j_\nu}^{\beta_\nu} \big|
$$
satisfies the inequality
$$
\frac{1}{n} E(\Sigma) \leq \sum_{\beta\in\cB}  \prod_{i: \beta_i=1} \frac{1}{n_i} E(\Sigma_i).
$$
Equality holds for $\cB=\{(1,1,\ldots,1)\}$ (the strong product) but, assuming no $\Sigma_i$ is without edges, in no other case.
\end{thm}

The form of the energy bound suggests that the average energy per vertex, \ $\bar E(\Sigma) = E(\Sigma)/|V(\Sigma)|$, \ may be an important quantity.  

\begin{proof}
The proof of the first equation is almost exactly the same as that of the corresponding result for unsigned graphs, \cite[Theorem 2.21]{spec1}.  The difference is that we must pay attention to the edge signs.  By Lemma~\ref{ldisjoint} we need only consider the term of one \ $\beta$ \ at a time; thus let \ $\Sigma = \NEPS(\Sigma_1,\ldots,\Sigma_\nu; \beta)$.

The elements of \ $A(\Sigma)$ \ are indexed by pairs of \ $\nu$-tuples, \ $(j_1,j_2,\ldots,j_\nu)$ \ and \ $(k_1,k_2,\ldots,k_\nu)$ \ where \ $1 \leq j_i \leq m_i$ \ and \ $1 \leq k_i \leq n_i$, \ corresponding to vertices \ $\bu=(u_{1j_1}, u_{2j_2}, \ldots, u_{\nu j_\nu})$ \ and \ $\bv=(v_{1k_1}, v_{2k_2}, \ldots, v_{\nu k_\nu})$ \ of the product graph.  Let us write \ $a_{\bu\bv}$ \ for the corresponding element of \ $A(\Sigma)$.  \ Then \ $a_{\bu\bv} = \sigma_1(u_{1j_1}v_{1k_1})^{\beta_1}\sigma_2(u_{2j_2}v_{2k_2})^{\beta_2}\cdots\sigma_\nu(u_{\nu j_\nu}v_{\nu k_\nu})^{\beta_\nu}$ \ by the definition of the \Cvet product, where in each signed graph we define \ $\sigma(uv) = 0$ \ if \ $u$ \ and \ $v$ \ are not adjacent.  The Kronecker product of the adjacency matrices has \ $(\bu,\bv)$-element \ $a_{1;j_1k_1}^{\beta_1}a_{2;j_2k_2}^{\beta_2} \cdots a_{\nu;j_\nu k_\nu}^{\beta_\nu}$ \ by Equation~\eqref{ekronecker}.  These two expressions are equal because \ $a_{i;j_ik_i} = \sigma_i(u_{ij_i}v_{ik_i})$ \ by the definition of the adjacency matrix.

The eigenvalues follow from Lemma~\ref{leigen}.

The energy is immediate from the definition and the eigenvalue formula.  The bound is computed from the energy formula:  
\begin{align*}
E(\Sigma) &\leq \sum_{\beta\in\cB}  \sum_{j_1=1}^{n_1} \cdots \sum_{j_\nu=1}^{n_\nu}  \big| \lambda_{1j_1}^{\beta_1} \cdots \lambda_{\nu j_\nu}^{\beta_\nu} \big| \\
&= \sum_{\beta\in\cB}  \prod_{i: \beta_i=0} \sum_{j_i=1}^{n_i} 1 \cdot  \prod_{i: \beta_i=1} \sum_{j_i=1}^{n_i} |\lambda_{ij_i}| 
= \sum_{\beta\in\cB}  \prod_{i: \beta_i=0} n_i \cdot \prod_{i: \beta_i=1} E(\Sigma_i) \\
&= \sum_{\beta\in\cB}  n  \prod_{i: \beta_i=1} \frac{1}{n_i} E(\Sigma_i).
\end{align*}

If \ $\cB=\{\beta\}$, \ the eigenvalues are \ $\lambda_{j_1\cdots j_\nu} = \lambda_{1j_1}^{\beta_1} \cdots \lambda_{\nu j_\nu}^{\beta_\nu}$~; therefore, equality holds in the calculation of the energy bound.  The only such \ $\cB$ \ permitted by the definition of the \Cvet product is \ $\cB = \{(1,1,\ldots,1)\}$, \ which is the strong product.

Now suppose \ $|\cB|>1$ \ and every \ $\Sigma_i$ has at least one edge.  That implies the eigenvalues of \ $A_i$ \ are not all \ $0$.  \ Since the sum of the eigenvalues of \ $\Sigma_i$ \ is the trace of \ $A_i$, \ which is \ $0$, \ $\Sigma_i$ \ has both positive and negative eigenvalues.  
There exists an index \ $I$, \ $1\leq I\leq\nu$, \ such that not all \ $\beta\in\cB$ \ have the same value $\beta_I$~; let \ $\beta', \beta'' \in \cB$ \ such that \ $\beta'_I=0$ \ and \ $\beta''_I=1$.  \ Choose \ $j_1,j_2,\ldots, j_\nu$ \ so that \ $\lambda_{ij_i}>0$ for \ $i \neq I$ \ but \ $\lambda_{Ij_I}<0$.  \ Now consider the eigenvalue \ $\pi = \lambda_{j_1j_2\ldots j_\nu}(\Sigma)$.  \ In its representation as a sum of terms \ $\prod_{i=1}^\nu \lambda_{ij_i}^{\beta_i}$ \ there are a positive term due to \ $\beta'$ \ and a negative term due to \ $\beta''$.  \ Therefore \ $\big| \sum_{\beta\in\cB} \lambda_{1j_1}^{\beta_1} \cdots \lambda_{\nu j_\nu}^{\beta_\nu} \big| < \sum_{\beta\in\cB} \big| \lambda_{1j_1}^{\beta_1} \cdots \lambda_{\nu j_\nu}^{\beta_\nu} \big|.$  \ It follows that the energy bound is strict.
\end{proof}

\begin{cor}\mylabel{t3} \quad 
The adjacency matrix \ $A(\Sigma)$ \ of  the Cartesian product \ $\Sigma = \Sigma_1 \times \cdots \times \Sigma_\nu$ \ of $\nu$ signed graphs is \ 
$$
A_1 \otimes I_{n_2} \otimes \cdots \otimes I_{n_\nu} + I_{n_1} \otimes A_2 \otimes \cdots \otimes I_{n_\nu} + \cdots + I_{n_1} \otimes I_{n_2} \otimes \cdots \otimes A_\nu.
$$

The eigenvalues of \ $\Sigma$ \ are the sums of the eigenvalues of the \ $\Sigma_i$~; i.e., \ 
$$
\lambda_{j_1j_2\ldots j_\nu}(\Sigma) = \lambda_{j_1}(\Sigma_1) + \cdots + \lambda_{j_\nu}(\Sigma_\nu).
$$

The energy \ $E(\Sigma)$ \ is given by the formula 
$$
E(\Sigma) = \sum_{j_1=1}^{n_1} \cdots \sum_{j_\nu=1}^{n_\nu} |\lambda_{j_1}(\Sigma_1) + \cdots + \lambda_{j_\nu}(\Sigma_\nu)|
$$ 
and satisfies the inequality \ 
$\frac{1}{n} E(\Sigma) \leq  \frac{1}{n_1} E(\Sigma_{1}) + \cdots + \frac{1}{n_\nu} E(\Sigma_\nu) ,$ \ 
where \ $n = n_1\cdots n_\nu$, \ with strict inequality if \ $\nu\geq2$ \ and at least two of the \ $\Sigma_i$ \ are not edgeless.
\end{cor}

\begin{proof} 
The corollary is immediate from Theorem~\ref{tadj} except the criterion for strict inequality, which is slightly stronger than that of Theorem~\ref{tadj} and is proved similarly.
\end{proof}

\subsection{Balance of the \Cvet product and the Cartesian product}

The eigenvalues provide a short proof that the Cartesian product is balanced if and only if all constituents are balanced.  Balance is important because, by Acharya's theorem, it causes the eigenvalues and energy (of both adjacency and Laplacian matrices) to be identical to those of the underlying unsigned graph, and therefore not interesting for themselves.  We begin with a general theorem that provides a sufficient but not a necessary condition for balance of a \Cvet product.

\begin{thm}\mylabel{tbal} \quad 
A \Cvet product \ $\Sigma = \NEPS(\Sigma_1,\ldots,\Sigma_\nu;\cB)$ \ is balanced if \ $\Sigma_{1},\ \ldots,\ \Sigma_\nu$ \ are all balanced.

Conversely, suppose \ $\cB$ \ contains the vector \ $\beta_i = (0,\ldots,0,1,0,\ldots,0)$ \ with \ $1$ \ in the \ $i$th position and \ $0$ \ elsewhere.  If \ $\Sigma_i$ \ is unbalanced, \ $\Sigma$ \ is also unbalanced.
\end{thm}

\begin{proof}
If all \ $\Sigma_i$ \ are balanced, then Theorem~\ref{t1} says that they have the same spectra as do their underlying graphs \ $G_i$.  \ The formulas (in Theorem~\ref{tadj}) for the eigenvalues of \ $\Sigma$ \ in terms of those of the \ $\Sigma_i$ \ and of its underlying graph \ $G = \NEPS(G_1,\ldots,G_\nu\nu;\cB)$ \ (regarded as all positive) in terms of the \ $G_i$ \ are exactly the same, so \ $\Sigma$ \ and \ $G$ \ have the same spectrum.  By Theorem~\ref{t1} again, \ $\Sigma$ \ is balanced.

The \Cvet product \ $\NEPS(\Sigma_1,\ldots,\Sigma_\nu;\beta_i)$ \ consists of \ $n/n_i$ \ copies of \ $\Sigma_i$, \ where \ $n = n_1\cdots n_\nu$.  \ Therefore \ $\Sigma_i$ \ is a subgraph of \ $\Sigma$.  \ A subgraph of a balanced graph is balanced, so if \ $\Sigma_i$ \ is unbalanced, \ $\Sigma$ \ is unbalanced.
\end{proof}

The first part of Theorem~\ref{tbal} does not have a general converse.  A counterexample is \ $\Sigma = \NEPS(-G_1,+K_2;\cB_2)$, \ where \ $\cB_2 = \{(1,1)\}$.  \ $\Sigma$ \ is bipartite and all negative; therefore it is always balanced.  However, \ $-G_1$ is balanced only when \ $G_1$ \ is bipartite.  It is an open problem to determine which bases \ $\cB$ \ have the property that for every \Cvet product \ $\Sigma = \NEPS(\Sigma_1,\ldots,\Sigma_\nu;\cB)$ \ with basis \ $\cB$, \ $\Sigma$ \ is balanced if and only if all the factors \ $\Sigma_i$ \ are balanced.  There is one important case in which there is such an if-and-only-if theorem.

\begin{thm}\mylabel{t4} \quad 
The following three statements about the Cartesian product \ $\Sigma = \Sigma_{1} \times \cdots \times \Sigma_\nu$ \ are equivalent.
\begin{enumerate}
\item [\rm{(i)}] $\Sigma$ \ is balanced.
\item [\rm{(ii)}] All of \ $\Sigma_{1},\ \ldots,\ \Sigma_\nu$ \ are balanced.
\item [\rm{(iii)}] $\Sigma$ \ and its underlying graph \ $G$ \ have the same spectrum.
\end{enumerate}

When \ $\Sigma$ is balanced, it and \ $G$ \ have the same energy.
\end{thm}

\begin{proof}
Theorem~\ref{t1} implies the equivalence of (i) and (iii).  Theorem~\ref{tadj} shows that (i) implies (ii).  We need only prove the converse.  The basis for the Cartesian product is \ $\cB_1$, \ which contains every unit vector \ $\beta_i$.  \ Therefore, balance of the Cartesian product implies balance of each \ $\Sigma_i$ \ by Theorem~\ref{tadj}.

The last part follows from  the definition of energy.
\end{proof}

\begin{cor}\mylabel{cbalancedcomp} \quad 
Let \ $\Sigma = \Sigma_{1} \times \cdots \times \Sigma_\nu$.  \ Then \ $b(\Sigma) = b(\Sigma_1)\cdots b(\Sigma_\nu)$.
\end{cor}

\begin{proof}
Each component of the Cartesian product is the Cartesian product of components \ $\Sigma_i'$ \ of \ $\Sigma_i$ \ for \ $1 \leq i \leq \nu$.  \ The component is balanced if and only if all \ $\Sigma_i'$ are balanced.
\end{proof}

\subsection{Laplacian matrix, eigenvalues and energy of the Cartesian product}
The formula for the Laplacian matrix of the Cartesian product is like that for the adjacency matrix.  We write \ $\lambda_i^L$ \ and \ $\mu_j^L$ \ for the Laplacian eigenvalues of \ $\Sigma_1$ \ and \ $\Sigma_2$, \ respectively.

\begin{thm}\mylabel{tdeg} \quad 
The degree matrix of a \Cvet product \ $G = \NEPS(G_1,\ldots,G_\nu;\cB)$ \ of graphs \ $G_i$ \ of order \ $n_i$, \ $1 \leq i \leq \nu$, \ is
$$
D(G) = \sum_{\beta\in \cB}  D(G_1)^{\beta_1} \otimes \cdots \otimes D(G_\nu)^{\beta_\nu} .
$$
The average degree of the product is \ 
$\bar d(G) = \sum_{\beta\in \cB} \prod_{i=1}^\nu \bar d(G_i)^{\beta_i}.$

In particular, the degree matrix of the Cartesian product is 
\begin{align*}
D(G_1\times\cdots G_\nu) &= D(G_1) \otimes I_{n_2} \otimes \cdots \otimes I_{n_\nu} + I_{n_1} \otimes D(G_2) \otimes \cdots \otimes I_{n_\nu} + \\
&\qquad \cdots + I_{n_1} \otimes I_{n_2} \otimes \cdots \otimes D(G_\nu).
\end{align*}
The average degree is \ $\bar d(G) = \sum_{i=1}^\nu \bar d(G_i)$.
\end{thm}

\begin{proof}
We evaluate the degree matrix of the product \ $G_\beta = \NEPS(G_1,\ldots,G_\nu;\beta)$.  \ By Lemma~\ref{ldisjoint} the degree matrix of \ $G$ \ is the sum of \ $D(G_\beta)$ \ over all \ $\beta \in \cB$.

By definition the neighbors of a vertex \ $\bu=(u_1,\ldots,u_\nu)$ \ are the vertices \ $\bv=(v_1,\ldots,v_\nu)$ \ such that \ $u_i=v_i$ \ if \ $\beta_i=0$ \ and \ $u_iv_i \in E(G_i)$ \ if \ $\beta_i=1$.  \ The elements of \ $\bv$ \ for which \ $\beta_i=0$ \ are fixed to be \ $u_i$, \ and those \ $v_i$ \ for which \ $\beta_i=1$ \ may independently vary over all neighbors of \ $u_i$.  \ Therefore the degree of \ $\bu$ \ is the product of the degrees \ $d_{G_i}(u_i)$ \ over all \ $i$ \ such that \ $\beta_i=1$.

The matrix \ $D(G_1)^{\beta_1} \otimes \cdots \otimes D(G_\nu)^{\beta_\nu}$ \ is diagonal and has as its diagonal element indexed by \ $\bu$ \ the number \ $d_{G_1}(u_1)^{\beta_1} \cdots d_{G_\nu}(u_\nu)^{\beta_\nu}$.  \ This is exactly \ $d_G(\bu)$.  \ Thus the formula for the degree matrix is proved.

Let \ $n=n_1\cdots n_\nu$.  \ As there are \ $n$ \ vertices, the average degree is determined by the equation
\begin{align*}
n \cdot \bar d(G_\beta) &= \sum_{j_1=1}^{n_1} \cdots \sum_{j_\nu=1}^{n_\nu} d_G(u_1,\ldots,u_\nu) \\
&= \sum_{j_1=1}^{n_1} d_{G_1}(u_1)^{\beta_1} \cdots \sum_{j_\nu=1}^{n_\nu} d_{G_\nu}(u_\nu)^{\beta_\nu} \\
&= \big[ n_1 \bar d(G_1)^{\beta_1} \big]  \cdots  \big[ n_\nu \bar d(G_\nu)^{\beta_\nu} \big]
\end{align*}
because \ $\sum_{j_i=1}^{n_i} d_{G_i}(u_i)^0 = n_i$ \ and \ $\sum_{j_i=1}^{n_i} d_{G_i}(u_i)^1 = n_i \bar d(G_i)$.  \ The left side of this equation is the total degree of \ $G_\beta$.  \ By edge-disjointness of the graphs \ $G_\beta$, \ the sum over \ $\beta \in \cB$ \ is the total degree of \ $G$.  \ Upon dividing by \ $n$ \ we get the desired formula.
\end{proof}

\begin{thm}\mylabel{tlap} \quad 
Given signed graphs \ $\Sigma_1$ \ of order \ $n_1$,  \ldots, \ $\Sigma_\nu$ \ of order \ $n_\nu$, \ the Laplacian matrix of the Cartesian product \ $\Sigma = \Sigma_{1} \times \cdots \times \Sigma_\nu$ \ is 
$$
L(\Sigma) = L(\Sigma_1)\otimes I_{n_2} \otimes \cdots \otimes I_{n_\nu} + \cdots + I_{n_1} \otimes I_{n_2} \otimes \dots \otimes L(\Sigma_\nu).
$$
The Laplacian eigenvalues of the Cartesian product are the sums of those of all the factors \ $\Sigma_i$~; i.e., \ $\lambda_{jk}^L(\Sigma) = \sum_{i=1}^\nu \lambda_{ij}^L$.
\end{thm}

\begin{proof}
The theorem follows from Theorem~\ref{tdeg}, Lemma~\ref{l1} and the formula of Corollary~\ref{t3} for the adjacency matrix.  By definition, 
\begin{align*}
L(\Sigma) &= D (\Sigma) - A (\Sigma) \\
&= \sum_{i=1}^\nu I_{n_1} \otimes \cdots \otimes D(\Sigma_i) \otimes \cdots \otimes I_{n_\nu} - \sum_{i=1}^\nu I_{n_1} \otimes \cdots \otimes A(\Sigma_i) \otimes \cdots \otimes I_{n_\nu} \\
&= \sum_{i=1}^\nu I_{n_1} \otimes \cdots \otimes \big[ D(\Sigma_i)-A(\Sigma_i) \big] \otimes \cdots \otimes I_{n_\nu}  \\
&= \sum_{i=1}^\nu I_{n_1} \otimes \cdots \otimes L(\Sigma_i) \otimes \cdots \otimes I_{n_\nu} .
\qedhere
\end{align*}
\end{proof}

Theorem~\ref{tlap} does not generalize to other \Cvet products.  For a vector \ $\beta$ \ of weight \ $\sum_{i=1}^\nu \beta_i > 1$, \ the Laplacian of \ $\Sigma_\beta = \NEPS(\Sigma_1,\ldots,\Sigma_\nu; \beta)$ \ is 
$$
D(\Sigma_\beta)-A(\Sigma_\beta) = D(\Sigma_1)^{\beta_1} \otimes \cdots \otimes D(\Sigma_\nu)^{\beta_\nu} - A(\Sigma_1)^{\beta_1} \otimes \cdots \otimes A(\Sigma_\nu)^{\beta_\nu} 
$$
will not combine by linear combination into a product of Laplacian matrices.  The general product \ $\NEPS(\Sigma_1,\ldots,\Sigma_\nu; \cB)$ \ where \ $\cB$ \ contains a vector of weight \ $>1$ \ has the same difficulty.  The only \Cvet product in which no vector has weight \ $> 1$ \ is the Cartesian product.

Theorem~\ref{tlap} implies the value of the Laplacian energy.  

\begin{cor}\mylabel{clapenergy} \quad
The Laplacian energy of \ $\Sigma = \Sigma_{1} \times \cdots \times \Sigma_\nu$ \ is given by 
\begin{align*}
E_L(\Sigma) &= \sum_{j_1=1}^{n_1} \cdots \sum_{j_\nu=1}^{n_\nu} \Big| \sum_{i=1}^\nu \lambda_{ij_i}^L - \bar d(\Sigma) \Big| ,
\end{align*}
which has an upper bound given by
\begin{align*}
\frac{1}{n} E_L(\Sigma) &\leq \sum_{i=1}^\nu  \frac{1}{n_i} E_L(\Sigma_i) .
\end{align*}
The inequality is strict unless \ $\nu=1$ \ or at most one of the \ $\Sigma_i$ \ has any edges.
\end{cor}

The average Laplacian energy per vertex, \ $\bar E_L(\Sigma) = \frac{1}{n} E_L(\Sigma)$, \ like the average energy per vertex mentioned at Theorem~\ref{tadj}, appears to be significant.

\begin{proof}
The first part of the corollary is an immediate consequence of the definition of Laplacian energy and Theorem~\ref{tlap}.  The second part follows from the formula of Theorem~\ref{tdeg} for the average degree of \ $\Sigma$ \ and the consequent inequality \ $\big| \sum_{i=1}^\nu \lambda_{ij_i}^L - \bar d(\Sigma) \big| \leq \sum_{i=1}^\nu |\lambda_{ij_i}^L - \bar d(\Sigma_i)|.$

The argument for strict inequality is similar to that in Theorem~\ref{tadj}.  The Laplacian energy of a signed graph \ $\Sigma$ \ of order \ $n$ \ with underlying graph \ $G$ \ is the trace of the matrix \ $L(\Sigma) - \bar d(G) I_n$.  \ The argument applied to the adjacency matrix in Theorem~\ref{tadj} should be applied to \ $L(\Sigma) - \bar d(G) I_n$ \ here.
\end{proof}

\subsection{Line graph}\mylabel{lg}

The general theorem on eigenvalues and energies of the line graph is well known for unsigned graphs.  For signed graphs it requires a new definition, namely, that of the line graph of a signed graph.  

The {\it line graph} of \ $\Sigma$ \ (\cite{tz2, LSD}) is the signed graph \ $\Lambda(\Sigma) = (\Lambda(G),\sigma_\Lambda)$, \ where \ $\Lambda(G)$ \ is the ordinary line graph of the underlying graph\footnote{We do not use the customary letter $L$ because we have used it for the Laplacian or Kirchhoff matrix.} and \ $\sigma_\Lambda$ \ is a signature such that every cycle in \ $\Sigma$ \ becomes a cycle with the same sign in the line graph, and any three edges incident with a common vertex become a negative triangle in the line graph.  The adjacency matrix is \ $A\big(\Lambda(\Sigma)\big) = 2I_m - \Eta(\Sigma)\transpose \Eta(\Sigma)$.  

There are two homogeneous special cases.  With an all-negative signature, \ $\Lambda(-G) = - \Lambda(G)$, \ so the all-negative signature is what gives the line graph of an unsigned graph.  For an all-positive signature, though, \ $\Lambda(+G)$ \ is not \ $+\Lambda(G)$ \ unless \ $G$ \ is bipartite with maximum degree at most \ $2$.  \ The Laplacian of \ $\Lambda(-G)$ \ is therefore the ``signless Laplacian'' of \ $\Lambda(G)$.

We can deduce the eigenvalues of the line graph from the Laplacian eigenvalues of the graph (Theorem~\ref{tlg}), but we cannot obtain the Laplacian eigenvalues of the line graph, except in the special case where the original underlying graph is regular (Theorem~\ref{tlgregular}.

\begin{thm}\mylabel{tlg} \quad 
Let \ $\Sigma$ \ be a signed graph of order \ $n$ \ and size \ $m$, \ with Laplacian eigenvalues \ $\lambda_1^L,\ldots,\lambda_{n-b(\Sigma)}^L >0$ \ and \ $\lambda_{n-b(\Sigma)+1}^L,\ldots,\lambda_n^L = 0$.

The eigenvalues of the line graph \ $\Lambda(\Sigma)$ \ are \ $2-\lambda_1^L,\ldots,2-\lambda_{n-b(\Sigma)}^L < 2$ \ and eigenvalue \ $2$ \ with multiplicity \ $m-n+b(\Sigma)$.  \ Its energy is 
$$
E(\Lambda(\Sigma)) = \sum_{i=1}^{n-b(\Sigma)}  | \lambda_i^L-2 | + 2\big(m-n+b(\Sigma)\big).
$$  
\end{thm}

\begin{proof}
The eigenvalues of the line graph are obtained by the standard method.  
The line graph \ $\Lambda = \Lambda(\Sigma)$ \ has adjacency matrix \ $2I_m - \Eta(\Sigma)\transpose \Eta(\Sigma)$.  \ The eigenvalues of \ $\Eta(\Sigma)\transpose \Eta(\Sigma)$ \ are the same as those of \ $L(\Sigma) = \Eta(\Sigma) \Eta(\Sigma)\transpose$, \ except that the multiplicity of \ $0$ \ changes from \ $b(\Sigma)$ \ to \ $m-n+b(\Sigma)$.  \ Thus, the adjacency matrix of \ $\Lambda$ \ has eigenvalues \ $2-\lambda_1^L,\dots,2-\lambda_{n-b(\Sigma)}^L$, \ and also \ $2$ \ with multiplicity \ $m-n+b(\Sigma)$.  \ That implies the value of the energy \ $E(\Lambda)$.
\end{proof}

\begin{thm}\mylabel{tlgregular} \quad 
Assume \ $\Sigma$ \ has underlying graph \ $G$ \ which is regular of degree \ $k$.  \ Let the eigenvalues of \ $\Sigma$ \ be \ $\lambda_1,\ldots,\lambda_{b(-\Sigma)} = -k$, \ $-k < \lambda_{b(-\Sigma)+1},\ldots,\lambda_{n-b(\Sigma)} < k$, \ and \ $\lambda_{n-b(\Sigma)+1},\ldots,\lambda_n = k$.

The line graph \ $\Lambda(\Sigma)$ \ has eigenvalues \ $\lambda_1-k+2,\ldots,\lambda_{b(-\Sigma)}-k+2 = -(2k-2)$, \ $-(2k-2) < \lambda_{b(-\Sigma)+1}-k+2,\ldots,\lambda_{n-b(\Sigma)}-k+2 < 2$ \ and eigenvalue \ $2$ \ with multiplicity \ $m-n+b(\Sigma)$.  \ 

Its Laplacian eigenvalues are \ $3k-4-\lambda_1,\ldots,3k-4-\lambda_{b(-\Sigma)} = 4k-4$, \ $4k-4 > 3k-4-\lambda_{b(-\Sigma)+1},\ldots,3k-4-\lambda_{n-b(\Sigma)} > 2k-4$, \ and \ $2k-4$ \ with multiplicity \ $m-n+b(\Sigma)$.  \ 

Its energy and Laplacian energy are: 
$$
E(\Lambda(\Sigma)) = E_L(\Lambda(\Sigma)) = \sum_{i=1}^{n-b(\Sigma)}  | \lambda_i-(k-2) | + 2\big(m-n+b(\Sigma)\big).
$$
\end{thm}

\begin{proof}
The range of values of eigenvalues of \ $\Sigma$ \ is taken from Lemma~\ref{lregular}.

As the degree of a vertex in \ $\Sigma$ \ is \ $k$, \ the total number of edges is \ $m=\frac12 kn$.  \ The degree of a vertex in the line graph is \ $\bar d(\Lambda) = 2k-2$.  \ 

Lemma~\ref{lregular} shows that the Laplacian eigenvalues of \ $\Sigma$ \ satisfy \ $\lambda_i^L = k-\lambda_i$.  \ Therefore by Theorem~\ref{tlg} the eigenvalues of \ $\Lambda$ \ have the form \ $2-k+\lambda_i$ \ (for \ $i=1,\ldots,n$) \ and \ $m-n$ \ other eigenvalues equal to \ $2$.  \ 
Substitution in the formulas of Theorem~\ref{tlg} gives the eigenvalues and energy of the line graph.  By Lemma~\ref{lregular} the Laplacian energy equals the energy.  

The Laplacian eigenvalues satisfy \ $\lambda^L = 2k-2-\lambda = 2k-4+\lambda^\Eta$.  \ Thus their values are \ $3k-4-\lambda_i$ \ for \ $i=1,\ldots,n$ \ and the eigenvalue \ $2k-4$ \ with multiplicity \ $m-n$ \ in addition to its multiplicity amongst the values \ $3k-4-\lambda_i$.
\end{proof}

We now apply these results to the Cartesian product of any number of signed graphs.  
Note that in the exceptional cases where \ $m-n < 0$, \ the ``additional eigenvalues'' equal to a value $\alpha$ constitute a deduction from the previously stated multiplicity \ $b(\Sigma)$ \  of $\alpha$.  This can occur only in the rare case that some component of \ $\Sigma$ \ is a tree (and not always then); that is, if all the factors \ $\Sigma_i$ \ have isolated vertices with at most one exception which must have a tree component.

\begin{thm}\mylabel{tlgcartesian} \quad 
Let \ $\Sigma_1,\ldots,\Sigma_r$ \ be \ $r$ \ signed graphs, and let \ $\Sigma_i$ \ have order \ $n_i$, \ size \ $m_i$ \ and Laplacian eigenvalues \ $\lambda_{ij}^L$ \ (for \ $j=1,\ldots,n_i$).  \ Let \ $\Sigma = \Sigma_1 \times \cdots \times \Sigma_r$, \ the Cartesian product, of order \ $n=n_1\cdots n_r$, \  average degree \ $\bar d(\Sigma) = \sum_{i=1}^r \bar d(\Sigma_i)$ \ and size \ $m = \frac12 n \bar d(\Sigma)$. 

Then the line graph \ $\Lambda(\Sigma)$ \ has eigenvalues 
$$
\lambda_{j_1j_2\ldots j_r} = 2 - (\lambda_{1j_1}^L + \cdots + \lambda_{rj_r}^L),
$$
of which exactly \ $b(\Sigma) = b(\Sigma_1)\cdots b(\Sigma_r)$ \ are equal to \ $2$ \ and the remainder are \ $<2$, \ together with \ $m-n = n(\frac12 \bar d(\Sigma) - 1)$ \ additional eigenvalues equal to \ $2$.  \ 
Its energy is 
$$
E(\Lambda(\Sigma)) = \sum_{j_1=1}^{n_1} \cdots \sum_{j_r=1}^{n_r}  | \lambda_{1j_1}^L + \cdots + \lambda_{rj_r}^L - 2 | + 2(m-n) .
$$
\end{thm}

\begin{proof}
The size of \ $\Sigma$ \ follows from its average degree.  The average degree follows from the fact that the degree of a vertex \ $(v_{1j_1},\ldots,v_{rj_r})$ \ in \ $\Sigma$ \ equals the sum of the degrees of its component vertices, \ $\sum_{i=1}^r d(v_{ij_i})$.

By obvious extensions of Theorems~\ref{t3} and~\ref{tlap} the Laplacian eigenvalues of \ $\Sigma$ \ are the sums of the Laplacian eigenvalues of the factor graphs \ $\Sigma_i$.

The fact that \ $b(\Sigma) = b(\Sigma_1)\cdots b(\Sigma_r)$ \ is Corollary~\ref{cbalancedcomp}.

The formula for energy in Theorem~\ref{tlg} can be rewritten as \ $\sum_{i=1}^n  | \lambda_i^L-2 | + 2(m-n)$ \ because \ $\lambda_i^L=0$ \ for \ $i=b(\Sigma)+1,\ldots,n$.  \ Then \ $\lambda_i^L = \lambda_{j_1j_2\ldots j_r}^L$ \ because \ $\Sigma$ \ is a Cartesian product.

Theorem~\ref{tlgcartesian} now follows from Theorem~\ref{tlg}.
\end{proof}

\begin{thm}\mylabel{tlgcartesianregular} \quad 
In Theorem~\ref{tlgcartesian} let the underlying graph of each \ $\Sigma_i$ \ be \ $k_i$-regular for \ $i=1,\ldots,r$ \ and let \ $k=k_1+\cdots+k_r$.  \ Then the underlying graph of \ $\Sigma$ is \ $k$-regular and that of \ $\Lambda(\Sigma)$ \ is \ $2(k-1)$-regular.  

The line graph \ $\Lambda(\Sigma)$ \ has eigenvalues
$$
\lambda_{j_1j_2\ldots j_r} = 2 - k + (\lambda_{1j_1} + \cdots + \lambda_{rj_r}),
$$
for \ $1\leq j_1 \leq n_1, \ldots, 1\leq j_r \leq n_r$, \ of which exactly \ $b(\Sigma) = b(\Sigma_1)\cdots b(\Sigma_r)$ \ are equal to \ $2$ \ and the remainder are \ $<2$ \ and \ $\geq -2(k-1)$ \ (amongst which are exactly \ $b(-\Sigma) = b(-\Sigma_1)\cdots b(-\Sigma_r)$ \ equal to \ $-2(k-1)$), \ together with \ $n(k-2)$ \ additional eigenvalues equal to \ $2$.

It has energy, also equal to the Laplacian energy, given by:  
$$
E(\Lambda(\Sigma)) = E_L(\Lambda(\Sigma)) = \sum_{j_1=1}^{n_1} \cdots \sum_{j_r=1}^{n_r}  \big| (\lambda_{1j_1} + \cdots + \lambda_{rj_r}) + 2 - k \big| + 2(k-2)n .
$$

It has Laplacian eigenvalues
$$
\lambda_{j_1j_2\ldots j_r}^L = 3k - 4 - (\lambda_{1j_1} + \cdots + \lambda_{rj_r}) ,
$$
of which exactly \ $b(\Sigma_1)\cdots b(\Sigma_r)$ \ are equal to \ $2k-4$ \ and the remainder are $> 2k-4$ \ and \ $\leq 4k-6$ \ (of which exactly \ $b(-\Sigma) = b(-\Sigma_1)\cdots b(-\Sigma_r)$ \ are equal to \ $4k-6$), \ together with \ $n(k-2)$ \ additional eigenvalues equal to \ $2k-4$.  \ 
\end{thm}

\begin{proof}
We substitute in Theorem~\ref{tlgcartesian} the values \ $k$ \ for the average degree of \ $\Sigma$ \ and \ $2(k-1)$ \ for the average degree of \ $\Lambda = \Lambda(\Sigma)$.  \ Thus \ $m = \frac12 2(k-1)n$ \ and \ $m-n = (k-2)n$.  \ We modify the energy summation as explained in the proof of Theorem~\ref{tlgcartesian}.  That gives the eigenvalues and energy of the line graph.  The numbers of eigenvalues that take on the extreme values $-(2k-2)$ and $2$ are given by Lemma~\ref{lregular}.

The Laplacian eigenvalues satisfy \ $\lambda^L = 2(k-1)-\lambda$ \ where \ $\lambda$ \ ranges over the eigenvalues.
\end{proof}

\section{Examples}\mylabel{examples}

In the sequel, while computing the energy of the Cartesian products of  certain signed graphs, we give emphasis to unbalanced signed graphs because, as Theorems~\ref{t1} and~\ref{t4} imply, the Cartesian product of balanced signed graphs behaves exactly like the product of the unsigned underlying graphs. 

\subsection{Constituent signed graphs}\mylabel{essential}

The signed graphs which are the constituents of the Cartesian product examples are paths and cycles.  
We denote by \ $P_n^{(r)}$, \ where \ $0 \le r \le n-1$, \ signed paths of order \ $n$ \ and size \ $n-1$ \ with \ $r$ \ negative edges where the underlying graph is the path \ $P_n$.  \ Note that \ $P_n^{(0)} = + P_n$ \ and \ $P_n^{(n-1)} = - P_n$, \ in which all edges are positive or negative, respectively.  
Similar notations are adopted for signed cycles \ $C_n^{(r)}$ \ with \ $r$ \ negative edges for \ $0\le r \le n.$  

Observing the fact that signed paths and indeed all signed trees do not contain any cycles, we have from Theorem~\ref{t1}:

\begin{cor}\mylabel{ctree} \quad 
A signed tree has the same eigenvalues as the underlying unsigned tree.
\end{cor}

In particular, from known spectra we deduce eigenvalues that will be used later.  (For the eigenvalues see e.g.\ \cite{sw} for the path and positive cycle and for all cycles \cite{hh, gp, am}.  For the Laplacian eigenvalues see \cite{gh}.)  To express the Laplacian matrix we employ square matrices \ $J_{n;ij}$ \ of order \ $n$ \ whose only nonzero element is \ $1$ \ in position \ $(i,j)$.

\begin{lem}\mylabel{c1} \quad  The signed paths \ $P_n^{(r)}$, \ where \ $0 \le r \le n-1$, \ have average degree \ $2-\frac{2}{n}$.  \ The eigenvalues are:
\begin{equation*}
   \lambda_j = 2 \cos \frac{\pi j}{n+1}  \ \ \text{ for } \ j = 1, 2, \dots, n.
\end{equation*}
The Laplacian matrix is \ $L(P_n^{(r)}) = 2I_n - (J_{n;11}+J_{n;nn}) - A(P_n^{(r)}).$  \ The Laplacian eigenvalues are:
\begin{equation*}
   \lambda_j^L = 2\big( 1 + \cos \frac{\pi j}{n} \big)  \ \ \text{ for } \ j = 1, 2, \dots, n,
\end{equation*}
all of which are positive except that \ $\lambda_n^L = 0$.
\end{lem}

\begin{proof}
The eigenvalues are the same as the known eigenvalues of the unsigned path.  The endpoints of the path are \ $v_1$ \ and \ $v_n$.  \ Thus, the degree matrix is \ $2I$ \ with \ $1$ \ subtracted in the upper left and lower right corners.  That is, \ $D(P_n^{(r)}) = 2I_n - (J_{n;11}+J_{n;nn}).$  \ The Laplacian  matrix follows at once.
\end{proof}

\begin{lem}\mylabel{t2} \quad The eigenvalues \ $\lambda_j$ \ of \ $C_n^{(r)}$ \ for \ $j=1,2, \dots, n$ \ and \ $0 \le r \le n$ \ are given by
$$
\lambda_j = 2 \cos \dfrac{(2j-[r])\pi}{n}, 
$$
for \ $j = 1, 2, \dots, n.$  \ The Laplacian matrix is \ $L(C_n^{(r)}) = 2I_n - A(C_n^{(r)}).$  \ The Laplacian eigenvalues are:
\begin{equation*}
\lambda_j^L = 2\big( 1 - \cos \dfrac{(2j-[r])\pi}{n} \big),
\end{equation*}
for \ $j = 1, 2, \dots, n,$ \ all of which are positive except that \ $\lambda_n^L = 0$ \ when \ $r$ \ is even.
 \end{lem}

\subsection{Product signed grids and ladders}
In light of Theorem~\ref{t4}, a signed (planar) grid \ $P_m^{(r_{1})}\times P_n^{(r_{2})}$ \ is balanced.  This graph has \ $nr_{1}+mr_{2}$ \ negative edges.  
Note that not all signatures of a grid \ $P_m \times P_n$ \ are balanced, but all signatures arising from Cartesian products are.  We refer to these signed graphs as {\it product signed graphs} to emphasize that they do not have arbitrary signs.

\begin{cor}\mylabel{cgrid} \quad  
A signed grid graph that is a Cartesian product of signed paths is balanced.  The adjacency matrix of a signed grid graph of the form \ $P_m^{(r_{1})}\times P_n^{(r_{2})},$ \ where \ $0\le r_1 \le m-1$ \ and \ $0\le r_2 \le n-1,$ \ is \ $A(P_m^{(r_{1})})\otimes I_{n} + I_{m}\otimes A(P_n^{(r_{2})})$ \ and the Laplacian matrix is 
$$
4I_{mn} - \big(A(P_n^{(r_{2})})+J_{m;11}+J_{m;mm}\big)\otimes I_n - I_m\otimes\big(A(P_n^{(r_{2})})+J_{n;11}+J_{n;nn}\big).
$$

The eigenvalues are given by:
\begin{equation*}
\lambda_{ij}(P_m^{(r_{1})}\times P_n^{(r_{2})})= 2 \Big(\cos\frac{\pi i}{m+1}+\cos\frac{\pi j}{n+1}\Big)
\end{equation*}
for \ $i = 1,2, \dots, m$ \ and \ $j = 1,2, \dots, n$, \ and the energy is:
\begin{equation*}
E(P_m^{(r_{1})}\times P_n^{(r_{2})})= 2\sum_{i=1}^{m}\sum_{j=1}^{n} \Big|\cos \frac{\pi i}{m+1} + \cos \frac{\pi j}{n+1} \Big|.
\end{equation*}

The Laplacian eigenvalues are given by:
\begin{align*}
\lambda_{ij}^L (P_m^{(r_{1})}\times P_n^{(r_{2})}) = 2\Big( 2 + \cos \frac{\pi i }{m} + \cos \frac{\pi j }{n} \Big) .
\end{align*}
Exactly one Laplacian eigenvalue is zero (that is \ $\lambda_{mn}^L$); \ the others are positive.  The Laplacian energy is given by:
\begin{equation*}
E_L(P_m^{(r_{1})}\times P_n^{(r_{2})}) = 2 \sum_{i=1}^{m} \sum_{j=1}^{n} \Big|\cos \frac{\pi i }{m} +\cos \frac{\pi j }{n}  - \frac{1}{m} - \frac{1}{n}\Big|.
        \end{equation*}
\end{cor}

\begin{proof}
The adjacency matrix follows from Corollary~\ref{t3} and the Laplacian follows from Theorem~\ref{tlap} and Lemma~\ref{c1}.

The eigenvalues follow from Corollary~\ref{t3} and Lemma~\ref{c1}.  (Note that by Theorem~\ref{t1} \ $P_m^{(r_{1})}\times P_n^{(r_{2})}$ \ has the same eigenvalues as the unsigned grid.)

The energy follows immediately from the definition.

We now look at the Laplacian eigenvalues and Laplacian energy.  By Theorem~\ref{tlap} the Laplacian eigenvalues are obtained by adding the Laplacian eigenvalues of \ $P_m^{(r_{1})}$ \ and those of \ $P_2^{(r_{2})}$.  Corollary~\ref{clapenergy} gives the Laplacian energy.

Noting the fact that the average degree is \ $4-\frac{2}{m}-\frac{2}{n}$ \ (from Theorem~\ref{tdeg} and Lemma~\ref{c1}), the Laplacian energy follows from Corollary~\ref{clapenergy}.
\end{proof}

The case \ $n=2$ \ is that of a signed ladder.  Here there is a slight simplification: the eigenvalues \ $2\cos \frac{\pi j}{n+1}$ \ are \ $\pm1$.

\subsection{Product signed cylindrical and toroidal grids}

In light of Theorem~\ref{t4}, the following signed graphs are unbalanced:
\begin{enumerate}
\item A signed cylindrical grid graph \ $C_m^{(r_{1})}\times P_n^{(r_{2})}$ \ when \ $r_1$ \ is odd, \ $0\le r_1 \le m$ \ and \ $0\le r_2 \le n-1.$  \ If \ $r_1$ \ is even, the cylindrical grid is balanced.
\item A signed toroidal grid graph \ $C_m^{(r_{1})}\times C_n^{(r_{2})}$ \ when \ $r_1$, \ $r_2$ \ or both are odd, \ $0\le r_1 \le m$ \ and \ $0\le r_2 \le n.$  \ If \ $r_1$ \ and \ $r_2$ \ are both even, the toroidal grid is balanced.
\end{enumerate}

A signed cylindrical grid that is the product of a signed cycle and a signed path is balanced or unbalanced depending on the parity of the number of negative edges in the signed cycle.  

\begin{cor}\mylabel{t8c} \quad 
A signed cylindrical grid of the form \ $C_m^{(r_{1})}\times P_n^{(r_{2})}$, \ where \ $0\le r_{1}\le m$ \ and \ $0\le r_{2}\le n-1$, \ has the adjacency matrix \ $A(C_m^{(r_{1})})\otimes I_{n} + I_{m}\otimes A(P_n^{(r_{2})})$.  \ The Laplacian matrix is 
$$
4I_{mn} - I_m\otimes(J_{n;11}+J_{n;nn}) - A(C_m^{(r_{1})})\otimes I_{n} - I_{m}\otimes A(P_n^{(r_{2})}).
$$
\end{cor}

\begin{proof}  The adjacency matrix follows from Corollary~\ref{t3}.  The Laplacian follows from Theorem~\ref{tlap} and Lemma~\ref{c1}.
\end{proof}

\begin{cor}\mylabel{ccylgrid} \quad 
A signed cylindrical grid of the form \ $C_m^{(r_{1})}\times P_n^{(r_{2})}$ \ where \ $0\le r_{1}\le m$ \ and \ $0\le r_{2}\le n-1$ \ has the eigenvalues:
$$
\lambda_{ij} = 2\Big( \cos\dfrac{(2i-[r_1])\pi}{m} + \cos\dfrac{2j\pi}{n+1} \Big), 
$$
for \ $1\le i \le m$ \ and \ $1\le j \le n$.  \ The energy is
$$
E(C_m^{(r_{1})}\times P_n^{(r_{2})}) = 2 \dsum_{i=1}^{m}\dsum_{j=1}^{n} \Big| \cos\dfrac{(2i-[r_1])\pi}{m} + \cos\dfrac{2j\pi}{n+1} \Big| .
$$

The Laplacian eigenvalues are given by:
$$
\lambda_{ij}^L = 2\Big( 2 - \cos\dfrac{(2i-[r_1])\pi}{m} + \cos\dfrac{2j\pi}{n} \Big), 
$$
for \ $1\le i \le m$ \ and \ $1\le j \le n$, of which all are positive, except for \ $\lambda_{mn}^L$ \ when \ $r_1$ \ is even.  The Laplacian energy is
$$
E_L(C_m^{(r_{1})}\times P_n^{(r_{2})}) = 2 \dsum_{i=1}^{m}\dsum_{j=1}^{n} \Big| 1-\frac{1}{n}-\cos\dfrac{(2i-[r_1])\pi}{m} + \cos\dfrac{2j\pi}{n} \Big| .
$$
\end{cor}

\begin{proof}  
The proof is the same as that of Corollary~\ref{cgrid} with slight changes.  The average degree is \ $4-\frac{2}{n}$ \ by Theorem~\ref{tdeg} and Lemma~\ref{c1}.
\end{proof}

A signed toroidal grid that is the product of two signed cycles is balanced or unbalanced depending on the numbers of negative edges in the constituent signed cycles.

\begin{cor}\mylabel{ctorgrid} \quad A signed toroidal grid of the form \ $C_m^{(r_{1})}\times C_n^{(r_{2})}$, \ where \ $0\le r_{1}\le m$ \ and \ $0\le r_{2}\le n$, \ has the adjacency matrix \ $A(C_m^{(r_{1})})\otimes I_{n} + I_{m}\otimes A(C_n^{(r_{2})})$ \ and the Laplacian matrix \ $4I_{mn} - A(C_m^{(r_{1})})\otimes I_{n} - I_{m}\otimes A(C_n^{(r_{2})})$.
\end{cor}

\begin{proof}  The adjacency matrix follows from Corollary~\ref{t3}.  The Laplacian follows from Theorem~\ref{tlap}.
\end{proof}

\begin{cor}\mylabel{ctorgride} \quad
 A signed toroidal grid of the form \ $C_m^{(r_{1})}\times C_n^{(r_{2})}$ \ where \ $0\le r_{1}\le m$ \ and \ $0\le r_{2}\le n$ \ has the eigenvalues:
$$
\lambda_{ij} = 2\Big( \cos\dfrac{(2i-[r_1])\pi}{m} + \cos\dfrac{(2j-[r_2])\pi}{n} \Big), 
$$
for \ $1\le i \le m$ \ and \ $1\le j \le n$.  
The Laplacian eigenvalues are:
$$
\lambda_{ij}^L = 4-2\Big( \cos\dfrac{(2i-[r_1])\pi}{m} + \cos\dfrac{(2j-[r_2])\pi}{n} \Big), 
$$
for \ $1\le i \le m$ \ and \ $1\le j \le n$.  \ The Laplacian eigenvalues are positive, except that \ $\lambda_{mn}^L=0$ \ when \ $r_1$ \ and \ $r_2$ \ are both even.

The energy and the Laplacian energy are both equal to
$$
2 \dsum_{i=1}^{m}\dsum_{j=1}^{n} \Big| \cos\dfrac{(2i-[r_1])\pi}{m} + \cos\dfrac{(2j-[r_2])\pi}{n} \Big|  .
$$
\end{cor}

\begin{proof}  
The eigenvalues follow from Corollary~\ref{t3}.  Then the Laplacian eigenvalues and energy follow from Lemma~\ref{lregular}.
\end{proof}

\subsection{Line graphs of product signed grids}\mylabel{lgt}

The line graphs of product signed grids can be treated by means of Theorem~\ref{tlgcartesian}.  That theorem does not give the Laplacian eigenvalues, so our results are limited, except for toroidal grids, which are regular and so are covered by Theorem~\ref{tlgcartesianregular}.

\begin{cor}\mylabel{clinegrid} \quad  
The line graph \ $\Lambda(P_m^{(r_{1})}\times P_n^{(r_{2})})$ \ of a signed grid graph that is a Cartesian product of signed paths of lengths \ $m,n\geq2$ \ has the eigenvalues:
\begin{equation*}
  \lambda_{ij} = 2 - 2 \Big(\cos\frac{\pi i}{m} + \cos\frac{\pi j}{n}\Big)
\end{equation*}
for \ $i = 1,2, \dots, m$ \ and \ $j = 1,2, \dots, n$, \ of which all are \ $< 2$ \ except that \ $\lambda_{mn}=2$, \ and also \ $2$ \ with additional multiplicity \ $(m-1)(n-1)-1$.  \ The energy is:
$$
E(\Lambda(P_m^{(r_{1})}\times P_n^{(r_{2})})) = 2\sum_{i=1}^{m}\sum_{j=1}^{n} \Big|-1 + \cos \frac{\pi i}{m} + \cos \frac{\pi j}{n} \Big| + 2(m-1)(n-1) - 2.
$$
\end{cor}

\begin{proof}
This is a corollary of Theorem~\ref{tlgcartesian} and either Lemma~\ref{c1} or Corollary~\ref{cgrid}.
\end{proof}

\begin{cor}\mylabel{clinecylalgrid} \quad 
The line graph \ $\Lambda(C_m^{(r_{1})}\times P_n^{(r_{2})})$ \ of a signed cylindrical grid where \ $n\geq2$, \ $0\le r_{1}\le m$ \ and \ $0\le r_{2}\le n-1$ \ has the eigenvalues:
\begin{align*}
\lambda_{ij} = 2\Big( \cos\dfrac{(2i-[r_1])\pi}{m} - \cos\dfrac{2j\pi}{n} - 1 \Big) ,
\end{align*}
of which all are \ $< 2$ \ except that \ $\lambda_{mn}=2$ \ when \ $r_1$ \ is even, and also the eigenvalue \ $2$ \ with additional multiplicity \ $m(n-1)$.  \ The energy is
\begin{align*}
E(\Lambda(C_m^{(r_{1})}\times P_n^{(r_{2})})) = 
2 \dsum_{i=1}^{m}\dsum_{j=1}^{n} \Big| 1 - \cos\dfrac{(2i-[r_1])\pi}{m} + \cos\dfrac{2j\pi}{n} \Big| + 2m(n-1).
\end{align*}
\end{cor}

\begin{proof}
This is another corollary of Theorem~\ref{tlgcartesian} and Lemma~\ref{c1} or Corollary~\ref{cgrid}.
\end{proof}

As toroidal grids are 4-regular, the line graphs of their product signatures fall within the scope of Theorem~\ref{tlgcartesianregular}.  

\begin{cor}\mylabel{clinetoroidalgrid} \quad 
The line graph \ $\Lambda(C_m^{(r_{1})}\times C_n^{(r_{2})})$ \ of a signed toroidal grid \ $C_m^{(r_{1})}\times C_n^{(r_{2})}$ \ where \ $0\le r_{1}\le m$ \ and \ $0\le r_{2}\le n$ \ has eigenvalues 
$$
\lambda_{ij} = 2\Big( \cos\dfrac{(2i-[r_1])\pi}{m} + \cos\dfrac{(2j-[r_2])\pi}{n} - 1 \Big), 
$$
all of which are \ $< 2$ \ except that \ $\lambda_{mn}=2$ \ if \ $r_1$ \ and \ $r_2$ \ are even, as well as \ $2$ \ with additional multiplicity \ $mn$.

The Laplacian eigenvalues are
$$
\lambda_{ij}^L = 8 - 2\Big( \cos\dfrac{(2i-[r_1])\pi}{m} + \cos\dfrac{(2j-[r_2])\pi}{n} \Big), 
$$
all of which are \ $>4$ \ except that \ $\lambda_{mn}^L=4$ \ if \ $r_1$ \ and \ $r_2$ \ are even, as well as \ $4$ \ with additional multiplicity \ $mn$.

The energy and the Laplacian energy are both equal to
$$
4mn + 2 \dsum_{i=1}^{m}\dsum_{j=1}^{n} \Big| \cos\dfrac{(2i-[r_1])\pi}{m} + \cos\dfrac{(2j-[r_2])\pi}{n} - 1 \Big|.
$$
\end{cor}

\begin{proof}
In this example of Theorem~\ref{tlgcartesianregular}, \ $k_1=k_2=2$ \ so \ $k=4$.  \ The product graph \ $\Sigma = C_m^{(r_{1})}\times C_n^{(r_{2})}$ \ has \ $mn$ \ vertices and \ $\frac12 kmn=2mn$ \ edges.  The eigenvalues of the line graph \ $\Lambda = \Lambda(\Sigma)$ \ are obtained from Theorem~\ref{tlgcartesianregular} and Corollary~\ref{ctorgride}.  The energies and Laplacian eigenvalues follow from Theorem~\ref{tlgcartesianregular}.
\end{proof}

\subsection{Line graphs of homogeneously signed graphs}\mylabel{lgh}

All-positive and all-negative signatures are interesting as they represent the ``signed'' and ``signless'' Laplacian matrices of ordinary graphs.  Recall that the Laplacian is \ $L(+G)$ \ and the ``signless'' Laplacian is \ $L(-G)$, \ so that the Laplacian eigenvalues of \ $G$ \ are the same as those of \ $+G$, \ i.e., \ $\lambda_i^L(G) = \lambda_i^L(+G)$, \ and the ``signless'' Laplacian eigenvalues are \ $\lambda_i^L(-G)$.  \ 
Also recall that \ $c(G)$ \ is the number of components of \ $G$ \ and \ $c_b(G)$ \ is the number of bipartite components of \ $G$.  \ 
Note that, since \ $L(+G) = D(G) - A(G)$ \ and \ $L(-G) = D(G) + A(G)$, \ the two Laplacian matrices are related by the identity \ $L(+G) + L(-G) = 2D(G)$.

\begin{cor}\mylabel{cline+} \quad 
Let \ $G$ \ be a graph of order \ $n$ \ and size \ $m$ \ with Laplacian eigenvalues \ $\lambda_1^L(+G),\lambda_2^L(+G),\ldots,\lambda_n^L(+G)$, \ of which \ $\lambda_1^L(+G),\ldots,\lambda_{n-c(G)}^L(+G) >0$ \ and \ $\lambda_{n-c(G)+1}^L(+G),\ldots,\lambda_n^L(+G) = 0$.

The eigenvalues of the line signed graph \ $\Lambda(+G)$ \ of \ $G$ \ with all positive signs are \ $2-\lambda_1^L(+G),\ldots,2-\lambda_{n-c(G)}^L(+G) < 2$ \ and eigenvalue \ $2$ \ with multiplicity \ $m-n+c(G)$.  \ Its energy is 
$$
E(\Lambda(+G)) = \sum_{i=1}^{n-c(G)}  | \lambda_i^L(+G)-2 | + 2\big(m-n+c(G)\big).
$$  
\end{cor}

\begin{proof}
In Theorem~\ref{tlg}, \ $\Sigma=+G$ \ and \ $b(\Sigma)=c(G)$.
\end{proof}

The invariants of \ $\Lambda(G)$ \ are identical to those of the all-postive signed line graph \ $+\Lambda(G)$.  \ The line graph \ $\Lambda(+G)$ \ is not in general all positive.  To get an all-positive signature of \ $\Lambda(G)$ \ we negate the line graph \ $\Lambda(-G)$, \ whose signature is all negative.

\begin{cor}\mylabel{cline-} \quad 
Let \ $G$ \ be a graph of order \ $n$ \ and size \ $m$ \ with underlying graph \ $G$ \ and signless Laplacian eigenvalues \ $\lambda_1^L(-G),\lambda_2^L(-G),\ldots,\lambda_n^L(-G)$, \ of which \ $\lambda_1^L(-G),\ldots,\lambda_{n-c_b(G)}^L(-G) >0$ \ and \ $\lambda_{n-c_b(G)+1}^L(-G),\ldots,\lambda_n^L(-G) = 0$.

The eigenvalues of the line signed graph \ $\Lambda(-G)$ \ are \ $2-\lambda_1^L(-G),\ldots,2-\lambda_{n-c_b(G)}^L(-G) < 2$ \ and eigenvalue \ $2$ \ with multiplicity \ $m-n+c_b(G)$.  \ Its energy is 
$$
E(\Lambda(-G)) = \sum_{i=1}^{n-c_b(G)}  | \lambda_i^L(-G)-2 | + 2\big(m-n+c_b(G)\big).
$$

The eigenvalues of the unsigned line graph \ $\Lambda(G)$ \ are \ $\lambda_1^L(-G)-2,\ldots,\lambda_{n-c_b(G)}^L(-G)-2 > -2$ \ and eigenvalue \ $-2$ \ with multiplicity \ $m-n+c_b(G)$.  \ Its energy equals \ $E(\Lambda(-G))$, \ thus it is a function of the signless Laplacian eigenvalues of \ $G$.
\end{cor}

\begin{proof}
In Theorem~\ref{tlg}, \ $\Sigma=-G$ \ and \ $b(\Sigma)=c_b(G)$.  \ Note that \ $A(\Lambda(G)) = A(+\Lambda(G)) = -A(-\Lambda(G)) = -A(\Lambda(-G))$, \ therefore the eigenvalues of \ $\Lambda(G)$ are the negatives of those of \ $\Lambda(-G)$.
\end{proof}

\subsection{Line graphs of homogeneously signed regular graphs}\mylabel{lghreg}

The line graphs of signed graphs whose underlying graphs are regular also fall under Theorem~\ref{tlgcartesianregular}.  

\begin{cor}\mylabel{cline+regular} \quad 
Let \ $G$ \ be a graph of order \ $n$ \ and size \ $m$ \ which is regular of degree \ $k>0$.  \ Let the eigenvalues of \ $G$ \ be \ $\lambda_1,\ldots,\lambda_{c_b(G)} = -k$, \ $-k < \lambda_{c_b(G)+1},\ldots,\lambda_{n-c(G)} < k$ \ and \ $\lambda_ {n-c(G)+1},\ldots,\lambda_n = k$.

The line signed graph \ $\Lambda(+G)$ \ has minimum eigenvalue \ $-(2k-2)$ \ with multiplicity \ $c_b(G)$, \ intermediate eigenvalues \ $-(2k-2) < \lambda_{c_b(G)+1}-k+2,\ldots,\lambda_{n-c(G)}-k+2 < 2$ \ and largest eigenvalue \ $2$ \ with multiplicity \ $m-n+c(G)$ \ (unless \ $G$ \ is a forest).

Its Laplacian eigenvalues are \ $4k-4$ \ with multiplicity \ $c_b(G)$, \ intermediate eigenvalues \ $4k-4 > 3k-4-\lambda_{c_b(G)+1},\ldots,3k-4-\lambda_{n-c(G)} > 2k-4$, \ and smallest eigenvalue \ $2k-4$ \ with multiplicity \ $m-n+c(G)$.

Its energy and Laplacian energy are: 
$$
E(\Lambda(+G)) = E_L(\Lambda(+G)) = \sum_{i=1}^{n-c(G)}  | \lambda_i-(k-2) | + 2\big(m-n+c(G)\big).
$$
\end{cor}

\begin{proof}
In Theorem~\ref{tlgregular}, \ $\Sigma=+G$, \ $b(\Sigma)=c(G)$ \ and \ $b(-\Sigma)=c_b(G)$.

It is well known that the only case in which \ $m-n+c(G) = 0$ \ is that in which \ $G$ \ is a forest.
\end{proof}

The line graph $\Lambda(G)$ is equivalent to $+\Lambda(G)$ but not in general to $\Lambda(+G)$ because the latter is not in general homogeneous.  To find the eigenvalues of $\Lambda(G)$ we work through $-\Lambda(-G)$, which is $\Lambda(-G)$ with all signs negated.

\begin{cor}\mylabel{cline-regular} \quad 
Let \ $G$ \ be a graph of order \ $n$ \ and size \ $m$ \ with underlying graph \ $G$ \ which is regular of degree \ $k>0$.  \ Let the eigenvalues of \ $G$ \ be \ $\lambda_1,\ldots,\lambda_{c_b(G)} = -k$, \ $-k < \lambda_{c_b(G)+1},\ldots,\lambda_{n-c(G)} < k$ \ and \ $\lambda_ {n-c(G)+1},\ldots,\lambda_n = k$.

The line signed graph \ $\Lambda(-G) = -\Lambda(G)$ \ has eigenvalues \ $-(2k-2)$ \ with multiplicity \ $c(G)$, \ $-(2k-2) < 2-k-\lambda_{c_b(G)+1},\ldots,2-k-\lambda_{n-c(G)} < 2$ \ and largest eigenvalue \ $2$ \ with multiplicity \ $m-n+c_b(G)$.  \ 
Its Laplacian eigenvalues, which are also the signless Laplacian eigenvalues of \ $\Lambda(G)$, \ are \ $4k-4$ \ with multiplicity \ $c(G)$, \ $4k-4 > 3k-4+\lambda_{c(G)+1},\ldots,3k-4+\lambda_{n-c_b(G)} > 2k-4$, \ and smallest eigenvalue \ $2k-4$ \ with multiplicity \ $m-n+c_b(G)$.  \ 
Its energy and Laplacian energy, which are also the signless Laplacian energy of \ $\Lambda(G)$, \ are: 
$$
E(\Lambda(-G)) = E_L(\Lambda(-G)) = \sum_{i=1}^{n-c_b(G)}  | \lambda_i+k-2 | + 2\big(m-n+c_b(G)\big).
$$

The eigenvalues of the unsigned line graph \ $\Lambda(G)$ \ are the negatives of those of $\Lambda(-G)$, i.e., \ $2k-2$ \ with multiplicity \ $c(G)$, \ $2k-2 > k-2+\lambda_{c_b(G)+1},\ldots,2k-2+\lambda_{n-c(G)} > -2$ \ and smallest eigenvalue \ $-2$ \ with multiplicity \ $m-n+c_b(G)$.  \ 
Its Laplacian eigenvalues are \ $0$ \ with multiplicity \ $c(G)$, \ $0 < k-\lambda_{c(G)+1},\ldots,k-\lambda_{n-c_b(G)} < 2k$, \ and largest eigenvalue \ $2k$ \ with multiplicity \ $m-n+c_b(G)$.  \ 
Its energy, Laplacian energy and signless Laplacian energy equal the energy of $\Lambda(-G)$.
\end{cor}

\begin{proof}
In Theorem~\ref{tlgregular}, \ $\Sigma=-G$, \ $b(\Sigma)=c_b(G)$ \ and \ $b(-\Sigma)=c(G)$.  \ By the identity $A(-G) = -A(G)$, \ $-G$ has the eigenvalues $-\lambda_i$.  Also, $\Lambda(G)$ has the same eigenvalues as $+\Lambda(G) = -(-\Lambda(G)) = -\Lambda(-G)$, so that the eigenvalues of $\Lambda(G)$ are the negatives of the eigenvalues of $\Lambda(-G)$.

The Laplacian eigenvalues of $\Lambda(-G)$ satisfy $\lambda^L(\Lambda(-G)) = 2k-2-\lambda(\Lambda(G))$ because $L(\Lambda(-G)) = (2k-2)I_m - A(\Lambda(-G))$.  These are the signless Laplacian eigenvalues of $\Lambda(G)$.  

The Laplacian eigenvalues of $\Lambda(G)$ are the eigenvalues of $L(-\Lambda(-G)) = (2k-2)I_m - A(-\Lambda(-G)) = (2k-2)I_m + A(\Lambda(-G))$.  Thus, they have the form $\lambda^L(\Lambda(G)) = 2k-2+\lambda(\Lambda(-G))$, which is as stated in the corollary.
\end{proof}

Of particular interest are the homogeneous signatures of \ $K_n$.  Recall that $K_n$ has eigenvalues $0$ with multiplicity $n-1$ and $n-1$ with multiplicity $1$.

\begin{cor}\mylabel{cline+complete} \quad 
The line signed graph \ $\Lambda(+K_n)$ \ has eigenvalues \ $3-n$ \ with multiplicity $n-1$ and \ $2$ \ with multiplicity \ $\binom{n-1}{2}$.

Its Laplacian eigenvalues are \ $3n-7$ \ with multiplicity $n-1$ and \ $2n-6$ \ with multiplicity \ $\binom{n-1}{2}$.

Its energy and Laplacian energy are: 
$$
E(\Lambda(+K_n)) = E_L(\Lambda(+K_n)) = (n-1)(2n-5).
$$
\end{cor}

\begin{proof}
Set $G=K_n$ in Corollary~\ref{cline+regular}.  Then $k=n-1$, $m=\binom{n}{2}$, $c_b(K_n)=0$ and $c(K_n)=1$.
\end{proof}

\begin{cor}\mylabel{cline-complete} \quad 
Let $n\geq3$.  

Let \ $K_n$ \ be a graph of order \ $n$ \ and size \ $m$ \ with underlying graph \ $K_n$ \ which is regular of degree \ $k>0$.  \ Let the eigenvalues of \ $K_n$ \ be \ $\lambda_1,\ldots,\lambda_{n-1} = 0$ \ and \ $\lambda_n = n-1$.

The line signed graph \ $\Lambda(-K_n) = -\Lambda(K_n)$ \ has eigenvalues \ $-2(n-2)$ \ with multiplicity \ $1$, \ $-(n-3)$ \ with multiplicity \ $n-1$, \ and \ $2$ \ with multiplicity \ $\binom{n-1}{2}-1$.  \ 
Its Laplacian eigenvalues, which are also the signless Laplacian eigenvalues of \ $\Lambda(K_n)$, \ are \ $4n-8$ \ with multiplicity \ $1$, \ $3n-7$ \ with multiplicity \ $n-1$, \ and \ $2n-6$ \ with multiplicity \ $\binom{n-1}{2}-1$.  \ 
Its energy and Laplacian energy are: 
$$
E(\Lambda(-G)) = E_L(\Lambda(-G)) = (n-1)(2n-5)+2(n-3).
$$

The eigenvalues of the unsigned line graph \ $\Lambda(K_n)$ \ are the negatives of those of $\Lambda(-K_n)$, i.e., \ $2(n-2)$ \ with multiplicity \ $1$, \ $n-3$ \ with multiplicity \ $n-1$, \ and \ $-2$ \ with multiplicity \ $\binom{n-1}{2}-1$.  \ 
Its Laplacian eigenvalues are \ $2$ \ with multiplicity \ $1$, \ $n+1$ \ with multiplicity \ $n-1$, \ and \ $2n$ \ with multiplicity \ $\binom{n-1}{2}-1$.  \ 

The energy, Laplacian energy, and signless Laplacian energy of $\Lambda(K_n)$ equal the energy of \ $\Lambda(-K_n)$.
\end{cor}

\begin{proof}
Set $G=K_n$ in Corollary~\ref{cline-regular}.
\end{proof}


\end{document}